\documentclass[11pt,a4paper]{article}

\usepackage{graphicx}
\usepackage{amsmath, nccmath}
\usepackage{amsfonts}
\usepackage{amssymb}
\usepackage{enumerate}
\usepackage{lscape}
\usepackage{longtable}
\usepackage{rotating}
\usepackage{color}
\usepackage{url}
\usepackage{rotating}
\usepackage{algpseudocode}
\usepackage[ruled]{algorithm}
\usepackage{pgf,tikz}
\usepackage{subfig}
\usepackage[titletoc]{appendix}

\usepackage{mathrsfs}
\usepackage{pgfplots}
\usetikzlibrary{arrows}
\usepackage{color}
\usepackage{booktabs}
\usepackage{multirow}
\usepackage{mathtools}
\usepackage{makecell}

\newtheorem{theorem}{Theorem}[section]
\newtheorem{proposition}{Proposition}[section]

\newtheorem{corollary}{Corollary}[section]

\newtheorem{assumption}{Assumption}[section]

\newenvironment{proof}[1][Proof]{\noindent \textbf{#1.} }{\hfill$\Box$\par\medskip}

\newcommand{\beqn}[1]{\begin{equation}\label{#1}}
\newcommand{\eeqn}{\end{equation}}

\definecolor{aau2}{rgb}{0.0, 0.5, 0.69}
\definecolor{aau3}{rgb}{0.0, 0.53, 0.74}
\definecolor{aau4}{rgb}{0.0, 0.48, 0.65}
\definecolor{aau5}{rgb}{0.0, 0.45, 0.73}
\definecolor{rsap}{RGB}{130, 36, 51}
\definecolor{gsap}{RGB}{112, 164, 137}

\definecolor{tud}{rgb}{0.43,0.73,0.11}
\definecolor{verde}{rgb}{0.33,0.53,0.11}

\definecolor{ttffqq}{rgb}{0.0, 0.48, 0.65} 
\definecolor{ffqqqq}{rgb}{0.0, 0.5, 0.69} 

\usetikzlibrary{arrows}

\tikzstyle{decision} = [diamond, draw, fill=blue!20,
    text width=4.5em, text badly centered, node distance=3cm, inner sep=0pt]
\tikzstyle{block} = [rectangle, draw, fill=blue!20,
     text centered, rounded corners, minimum height=4em]
\tikzstyle{line} = [draw, -latex']
\tikzstyle{cloud} = [draw, ellipse,fill=red!20, node distance=3cm,
    minimum height=2em]
\tikzstyle{cloud2} = [draw, ellipse,fill=green!20, node distance=3cm,
    minimum height=2em]
\begin{document}

\title{Convergence rates of the stochastic alternating algorithm for bi-objective optimization}

\date{December 28,  2022}

\author{
S. Liu\thanks{Department of Industrial and Systems Engineering,
Lehigh University,
200 West Packer Avenue, Bethlehem, PA 18015-1582, USA
({\tt sul217@lehigh.edu}).}\and
L. N. Vicente\thanks{Department of Industrial and Systems Engineering,
Lehigh University,
200 West Packer Avenue, Bethlehem, PA 18015-1582, USA and
Centre for Mathematics of the University of Coimbra (CMUC)
({\tt lnv@lehigh.edu}). Support for
this author was partially provided by the Centre for Mathematics of the University of Coimbra under grant FCT/MCTES UIDB/MAT/00324/2020.}
}

\maketitle

{\small
\begin{abstract}
Stochastic alternating algorithms for bi-objective optimization are considered when 
optimizing two conflicting functions for which optimization steps have to be applied separately for each function. Such algorithms consist of applying a certain number of steps of gradient or subgradient descent on each single objective at each iteration.
In this paper, we show that stochastic alternating algorithms achieve a sublinear convergence rate of $\mathcal{O}(1/T)$, under strong convexity, for the determination of a minimizer of a weighted-sum of the two functions, parameterized by the number of steps applied on each of them.  An extension to the convex case is presented for which the rate weakens to $\mathcal{O}(1/\sqrt{T})$. These rates are valid also in the non-smooth case. Importantly, by varying the proportion of steps applied to each function, one can determine an approximation to the Pareto front.
\end{abstract}

\bigskip

\begin{center}
\end{center}
}

\section{Introduction}
\label{introduction}

Numerous real-world application scenarios involve several potentially conflicting objectives, which need to be considered simultaneously. Such type of optimization problems is referred to as multi-objective
optimization (MOO). One can find various applications spanning across applied engineering, operations management, finance, economics and social sciences, agriculture, green logistics, and health systems~\cite{NGunantara_2018}. As these objectives are usually competing among each other, it is not possible to find a single solution which is optimal with respect to all objectives. The goal of MOO is to find a set of equally good solutions known as Pareto optimal solutions or efficient points. Roughly speaking for each Pareto optimal solution, there is no other point in the feasible region leading to a simultaneous improvement in all objectives. The determination of the set of Pareto optimal solutions helps decision makers to define the best trade-offs among the several competing criteria.

In this paper, we consider a continuous bi-objective optimization problem
\begin{equation}
\label{biobjective_smooth}
    \min~F(x) = (f^a(x), f^b(x))^\top, \quad x \in \mathcal{X},
\end{equation}
where $\mathcal{X}$ is a nonempty bounded closed convex region. 
We say that the multi-objective problem is smooth if all objective functions~$f^i$ are continuously differentiable. In the convergence theory of our algorithms, we will assume that both functions are convex. Furthermore, both objectives involve randomness in its parameters, in which case we expect that a true gradient or subgradient is not available or too expensive to compute. Instead, one can access a unbiased estimator $g^i(x, \xi), i \in \{a, b\}$, of a true gradient or subgradient of $f^i$, where $\xi$ denotes some random variable.
The optimality in MOO is defined via the concept of Pareto dominance. More precisely, for the general bi-objective problem~\eqref{biobjective_smooth}, we say that $x$ (weakly) dominates~$y$ if $f^i(x) \leq f^i(y), \forall i \in \{a, b\}$, and $F(x) \neq F(y)$. A point~$x \in \mathcal{X}$ is called a (strictly) non-dominated solution or Pareto optimal solution if it is not (weakly) dominated by any other point in~$\mathcal{X}$. The set of non-dominated solutions denoted by~$\mathcal{P}$ forms the so-called Pareto front $F(\mathcal{P})= \{F(x): x \in \mathcal{P}\}$. 

\subsection{A quick overview of multi-objective optimization methods}

Depending on whether the decision maker's preference is explicitly involved during the optimization process, two categories of approaches, namely \textit{a priori} and \textit{a posteriori} methods, are mainly considered in the existing multi-objective optimization literature. 

The \textit{a priori} methods incorporate the preference to reduce the multi-objective problem into a single-objective one, which can then be tackled by classical single-objective optimization algorithms. Such preferences could be expressed via weights put on different objectives, desired upper bounds of some objectives, or utility functions, corresponding to the weight-sum method~\cite{SGass_TSaaty_1955}, the $\epsilon$-constrained method~\cite{YVHaimes_1971}, and the utility function method~\cite{EKBrowning_MAZupan_2020,KMiettinen_2012}. In the weighted-sum method, one assigns each objective a non-negative weight $\lambda_i$ and then optimizes the weighted-sum function $S(x, \lambda) = \sum_i \lambda_i f^i(x)$. By varying the weights in a convex linear combination one is guaranteed to find all Pareto solutions when the functions are convex. The $\epsilon$-constrained method consists of optimizing one objective using given upper bounds, i.e., $\min~f^i(x)$ subject to $f^j(x) \leq \epsilon_j, \forall j \neq i$. The utility function method is a more general approach in this category. A utility function~$U$ is a real function quantifying the overall preference among the competing objectives~\cite{EKBrowning_MAZupan_2020}. Given two candidate solutions $x$ and $y$, then the decision maker prefers $x$ to $y$ if $U(F(x)) > U(F(y))$. 

As for the \textit{a posteriori} methods, the solution process aims at producing all the Pareto optimal solutions or a representative subset of the Pareto optimal solutions, out of which a single ``best'' solution is then selected according to the decision maker's preferences. A technique often used is to maintain and update a list of candidates iteratively based on a {\it heuristic} or {\it rigorous} mechanism. The class of (\textit{a posteriori}) evolutionary multi-objective optimization algorithms~\cite{SBechikh_RDatta_AGupta_2016, CCCoello_2006} using {\it metaheuristics} has been a very
popular research topic since 1990's. Moreover, in the past two decades, many {\it rigorous} descent methods~\cite{EHFukuda_LMGDrummond_2014} have been developed, namely multi-objective versions of the zero-order method~\cite{ALCustodio_etal_2011}, the steepest descent~\cite{JADesideri_2012,JADesideri_2014,LGDrummond_ANIusem_2004, LGDrummond_BFSvaiter_2005, JFliege_BFSvaiter_2000, EHFukuda_LMGDrummond_2013}, the subgradient methods~\cite{GCBento_JXCruzNeto_2013, JYBelloCruz_2013}, the proximal gradient~\cite{GCBento_etal_2014, HBonnel_ANIusem_BFSvaiter_2005, HTanabe_EHFukuda_NYamashita_2019, HTanabe_EHFukuda_NYamashita_2020}, the conjugate gradient methods~\cite{LRLucambioPerez_LFPrudente_2018}, the trust-region methods~\cite{SQu_MGoh_BLiang_2013,KDVillacorta_PROliveira_ASoubeyran_2014}, and the Newton methods~\cite{LGDrummond_FMPRaupp_BFSvaiter_2014, JFliege_LGDrummond_BFSvaiter_2009, ZPovalej_2014, SQu_MGoh_FTSChan_2011}. These methods attempt to decrease the individual objective values simultaneously in some sense, and the convergence of the generated iterates to a Pareto first-order stationary point is rigorously established under reasonable assumptions. Such a common descent direction for all the objectives is usually computed at each iteration by solving a certain sub-problem that takes into account the gradient and/or Hessian information of all the objectives. 

Stochastic multi-objective optimization (SMOO) deals with the case where one or multiple objective functions involve uncertainty or noisy data. One usually formulates the objective functions in the form of expectation with respect to some random variables and approximate them by the sample average approximation (SAA) technique. Two main approaches~\cite{FBAbdelaziz_1992, FBAbdelaziz_2012, RCaballero_2004} for solving the resulting SMOO problems are the \textit{multi-objective} methods and the \textit{stochastic} methods. While the former reduces the SMOO problem into a deterministic MOO problem, the later converts the original problem into a single-objective stochastic problem by scalarization. 

The alternating optimization algorithms considered in this paper belong to the \textit{a priori} category.
Instead of decreasing individual objectives simultaneously at each iteration, such algorithms separately optimize them, one after the other, and we will see in this paper that these algorithms eventually optimize a certain convex linear combination of the two functions. 
In terms of how the stochasticity is handled, our paper uses the stochastic approximation (SA) technique (stochastic gradient-based methods) which has been recently considered for SMOO by~\cite{SLiu_LNVicente_2019, MQuentin_PFabrice_JADesideri_2018}.

An example of stochastic alternating algorithms for bi-objective optimization arises in 
fair clustering, where one tries to capture a trade-off between clustering cost and balance among different demographic groups defined by sensitive attributes
(see~\cite{FChierichetti_etal_2017, SBera_etal_2019, MSchmidt_CSchwiegelshohn_CSohler_2018, IMZiko_etal_2019}). 
A perfectly balanced solution requires that each cluster has exactly the same proportions of each group, but such solution is not necessarily the one that minimizes the clustering cost.
Fair clustering can be formulated as a bi-objective problem~(see~\cite[Problem (2)]{SLiu_LNVicente_2021}), involving two non-convex and non-smooth functions. The problem is originally discrete but a continuous relaxation of the binary assignment variables can be considered. In~\cite{SLiu_LNVicente_2021}, we introduced a stochastic alternating balance fair $k$-means (SAfairKM) approach for fair clustering.
In this application we could not find a reasonable way to optimize a convex linear combination of the two functions. Rather, we were faced with a situation where we knew how to optimize them separately.

\subsection{Contribution of this paper}

In this paper, we establish rates of convergence for stochastic alternating bi-objective optimization algorithms, both in the smooth and non-smooth cases, and under both simple and strong convexity.
It is shown that the algorithms exhibit sublinear convergence rates of $\mathcal{O}(1/T)$ when determining a Pareto solution, under strong convexity and classical assumptions of gradients or subgradients. Moreover, we show that downgrading strong convexity to convexity results in a degradation of the convergence rate from $\mathcal{O}(1/T)$ to $\mathcal{O}(1/\sqrt{T})$. It is remarkable that we recover the same convergence rates of the corresponding single-objective stochastic gradient/subgradient methods. Our theory evolves around an ingenious application of the Intermediate Value Theorem (IVT) to aggregate the steps applied when separately optimizing each function. 
The target Pareto solution is defined by selecting a convex span of the number of steps applied to each function. By varying the weight in this convex span, one can determine the whole Pareto front in the convex case. It is the application of the IVT that brings the convex span of the effort into the convergence rates.

Methods like weighted-sum or $\epsilon$-constrained would be the reasonable choices in most practical applications whenever it is possible to optimize a convex linear combination of the objective functions or optimize one of them subject to the others be within a certain range. 
Our alternating methodology would be chosen when optimizing such a combination is not possible. That is the case when the functions are defined by some form of black-box for which there is a legacy or binary code to carry out the optimization. That is also the case when there is some form of function-dependent descent step~\cite{SLiu_LNVicente_2021}. Our methodology would also be chosen when optimizing such combination is possible but not desirable, which may happen for instance when the functions exhibit a scale that is very different and hard to balance.

\section{Alternating bi-objective optimization algorithms}

The key idea of the alternating bi-objective optimization for solving~\eqref{biobjective_smooth} consists of iteratively taking $n_a$ gradient (or subgradient) descent steps for the first objective and then $n_b$ gradient (or subgradient) descent steps for the second objective. For simplicity, we denote such $n_a + n_b$ steps as one single iteration of the algorithm. The stochastic alternating bi-objective gradient or subgradient algorithm (SA2GD) is formally described in Algorithm~\ref{alg1_SA2GD}. At every iteration $t$, starting at $x_t = y^a_{0, t}$, the algorithm computes two sequences of intermediate iterates, $\{y^a_{r, t}\}_{r=1}^{n_a}$ and $\{y^b_{r, t}\}_{r=1}^{n_b}$. Let $P_{\mathcal{X}}$ denote the orthogonal projection operator that projects 
the new iterate back to the feasible region. Such an operator is well-defined because we assume a compact and convex feasible region.

{\linespread{1.15}\addtocounter{algorithm}{-1}
\renewcommand{\thealgorithm}{{1}} 
\begin{algorithm}[ht]
\caption{Stochastic alternating bi-objective gradient (or subgradient) algorithm (SA2GD)} 
\label{alg1_SA2GD} 
\begin{algorithmic}[1]
\item \textbf{Input:} Initial point $x_0 = y^a_{0, 0}$ and a step size sequence $\{\alpha_t\}$.
\item \textbf{Output:} A likely non-dominated or Pareto solution $x_{T} = y^a_{0, T}$.
\item {\bf for} $t = 0, 1, \ldots, T-1$ {\bf do}
\item \quad {\bf for} $r = 0, \ldots, n_a-1$ {\bf do}
\item \quad \quad Generate a stochastic gradient $g^a(y^a_{r, t}, \xi^r_{t})$.
\item \quad \quad Update $y^a_{r+1, t} = y^a_{r, t} - \alpha_t g^a(y^a_{r, t}, \xi^r_{t})$. 
\item \quad Set $y^b_{0, t} = y^a_{n_a, t}$.
\item \quad {\bf for} $r = 0, \ldots, n_b - 1$ {\bf do}
\item \quad \quad Generate a stochastic gradient $g^b(y^b_{r, t}, \xi^{n_a + r}_{t})$.
\item \quad \quad Update $y^b_{r+1, t} = y^b_{r, t} - \alpha_t g^b(y^b_{r, t}, \xi^{n_a + r}_{t})$. 
\item \quad Set $x_{t+1} = y^a_{0, t+1} = P_{\mathcal{X}}(y^b_{n_b, t})$.
\par\vspace*{0.1cm}
\end{algorithmic}
\end{algorithm}
}

Similarly, in the case where both objectives are non-smooth but sub-differentiable, one can develop a counterpart stochastic alternating bi-objective subgradient method by simply replacing stochastic gradients used in Lines 5 and 9 by stochastic subgradients. 

The stochastic gradient algorithm for SMOO (stochastic multi-gradient) computes one point in the Pareto front (following a steepest descent principle)~\cite{SLiu_LNVicente_2019, MQuentin_PFabrice_JADesideri_2018}. The alternating one also computes just one point. The difference is that the stochastic multi-gradient method is not parameterized or scalarized, and so the only form to use it to compute the whole Pareto front is to incorporate it in a non-dominated list updating mechanism or some form of hotstarting or randomization. The alternating one works like the weighted-sum and $\epsilon$-constrained methods, and it involves parameters such that when varied allow us to capture the whole Pareto front (the optimization effort quantified by $n_a,n_b$).

\section{Convergence analysis}

Our convergence analysis starts with an easier case where both objectives are assumed smooth and strongly convex. When we remove smoothness from the assumptions, the same convergence rate is maintained. Furthermore, in both cases, relaxing strong convexity to convexity leads to degradation of convergence rates from $\mathcal{O}(1/T)$ to $\mathcal{O}(1/\sqrt{T})$. 

\subsection{The smooth and strongly convex case}
\label{proof_sa2gd}

We first assume that both objective functions are smooth and strongly convex and establish a $\mathcal{O}(1/T)$ sublinear convergence rate for the stochastic alternating bi-objective gradient descent (SA2GD) algorithm. At every iteration $t$, starting at $x_t = y^a_{0, t}$, SA2GD computes two sequences of intermediate iterates, $\{y^a_{r, t}\}_{r=1}^{n_a}$ and $\{y^b_{r, t}\}_{r=1}^{n_b}$. We first make a formal assumption on the boundedness of the feasible region.

\begin{assumption}
\label{ass:limit_pts}
\textbf{(Bounded feasible region)}
The feasible region $\mathcal{X}$ is compact, and in particular there exists a positive constant $\Theta$ such that
\begin{equation*}
\displaystyle \max_{x, \bar{x} \in \mathcal{X}} \| x - \bar{x}\| \;\leq\; \textstyle \Theta \;<\; \infty.
\end{equation*}
\end{assumption}

Then, we formalize the classical smoothness assumption of Lipschitz continuity of the gradients.

\begin{assumption}
\label{ass:Lipschitz}
\textbf{(Lipschitz continuous gradients)} The individual true gradients are Lipschitz continuous with Lipschitz constants $L_i > 0, i \in \{a, b\}$, i.e., 
\begin{equation*}
    \begin{split}
        \|\nabla f^i(x) - \nabla f^i(\bar{x}) \| \;\leq\;& L_i \| x - \bar{x} \|, \quad \forall (x, \bar{x}) \in \mathcal{X} \times \mathcal{X}.
    \end{split}
\end{equation*}
\end{assumption}

Note that any continuous function  attains a maximum on a compact set, and thus one can assume the existence of positive constants~$M_\nabla^i$, $i \in \{a,b\}$, such that
\begin{equation} \label{Mnabla}
\| \nabla f^i(x) \|^2 \; \leq \; M_\nabla^i,  \quad \forall x \in \mathcal{X}.  
\end{equation}
In addition to Assumption~\ref{ass:Lipschitz}, we impose strong convexity in both objective functions.

\begin{assumption}
\label{ass:strongconv}
\textbf{(Strong convexity)} For all $i = \{a, b\}$, there exists a scalar $c_i > 0$ such that

\begin{equation*}
f^i(\bar{x}) \;\geq\; f^i(x) + \nabla f^i(x)^\top(\bar{x} - x) + \frac{c_i}{2} \| \bar{x} - x \|^2, \quad \forall (x,\bar{x}) \in \mathcal{X} \times \mathcal{X}.
\end{equation*}
\end{assumption}

Note that based on the above assumption, the weighted-sum function $S(x, \lambda) = \lambda f^a(x) + (1-\lambda)f^b(x), \lambda \in [0, 1]$, is also strongly convex with constant $c = \min(c_a, c_b)$. Given the individual stochastic gradients $g^i(x_t, \xi_t), \forall i \in \{a, b\}$, generated with random variable~$\xi_t$, we use $\mathbb{E}_{\xi_t}[\cdot]$ to denote the conditional expectation taken with respect to $\xi_t$. We also impose the following two classical assumptions of stochastic gradients. 

\begin{assumption}
\label{ass:gradient}
For both objective functions $i \in \{a, b\}$, and all iterates $t \in \mathbb{N}$, the individual stochastic gradients~$g^i(x_t, \xi_t)$ satisfy the following:
\begin{enumerate}
    \item[(a)]\textbf{(Unbiased gradient estimation)} 
    $\mathbb{E}_{\xi_t}[g^i(x_t, \xi_t)] \;=\; \nabla f^i(x_t), ~\forall i \in \{a, b\}$. 
    \item[(b)]\textbf{(Bound on the second moment)}
There exist positive scalars $G_i > 0$ and $\bar{G}_i > 0$ such that
\begin{equation*}
\begin{split}
\mathbb{E}_{\xi_t}[ g^i(x_t, \xi_t)\|^2] \;\leq\; G_i + \bar{G}_i\| \nabla f^i(x_t)\|^2, \quad \forall i \in \{a, b\}.
\end{split}
\end{equation*}
\end{enumerate}
\end{assumption}

The above assumptions are the commonly used ones in classical stochastic gradient methods~\cite{LBottou_FECurtis_JNocedal_2018}, basically assuming reasonable bounds on the expectation and variance of the individual stochastic gradients. 

Our main result stated below shows that SA2GD drives the expected optimality gap of the weighted-sum function $S(\cdot, \lambda(n_a, n_b))$ to zero at a sublinear rate of $1/T$, where $\lambda(n_a, n_b) = n_a/(n_a + n_b)$, when using a decaying step size sequence. By varying $n_a$ and $n_b$ in $\{1, \ldots, n_{\text{total}}\}$, such that $n_{\text{total}} = n_a + n_b$, one can capture the entire trade-off between $f^a$ and $f^b$.

\begin{theorem} (\textbf{Sublinear convergence rate of SA2GD in the smooth and strongly convex case})
\label{th:smooth_strongcov}
Let Assumptions~\ref{ass:limit_pts}-\ref{ass:gradient} hold and $x_*$ be the unique minimizer of the weighted function $S(\cdot, \lambda_*)$ in~$\mathcal{X}$, where $\lambda_* = \lambda(n_a, n_b) = n_a/(n_a + n_b)$. Choosing a diminishing step size sequence
$\alpha_t = \frac{2}{c(t+1)(n_a + n_b)}$, the sequence of iterates generated by the SA2GD algorithm satisfies
\begin{equation*}
    \min_{t = 1, \ldots,T} \mathbb{E}[S(x_t, \lambda_*)] - S(x_*, \lambda_*) \;\leq\; \frac{4}{c(T + 1)} \left(\hat{G}^2 +  L\Theta\hat{G}\right).   
\end{equation*}
 where $\hat{G} = \sqrt{G + \bar{G}M_{\nabla}}$, $L = \max(L_a, L_b)$, $G = \max(G_a, G_b)$, $\bar{G} = \max(\bar{G}_a, \bar{G}_b)$, and $M_\nabla = \max(M_\nabla^a, M_\nabla^b)$.
\end{theorem}

\begin{proof}
The proof is divided in three parts for better organization and understanding. In the first part, one obtains an upper bound on the norm of the iterates, $\mathbb{E}_{\xi_t}[\|x_{t+1} - x_*\|^2]$. Strong convexity of the weighted-sum function is applied in the second part. The third part concludes the proof using standard arguments. For simplicity, we let $z_t = y^b_{0, t} = y^a_{n_a, t}$ be the intermediate point at each iteration.  \vspace{1ex}

\noindent {\bf Part I: Bound on the iterates error using the Intermediate Value Theorem.} 
At any iteration $t$, the sequence of stochastic gradients is computed from drawing the sequence of random variables $\xi_t = \{\xi^0_t, \ldots, \xi^{n_a + n_b - 1}_t \}$. We have
\begin{align*}
    x_{t+1} - x_* & \;=\; P_{\mathcal{X}}\left(z_t - \alpha_t \sum_{r = 0}^{n_b - 1} g^b(y^b_{r, t}, \xi^r_t)\right) - x_*, \\
    &\;=\;P_{\mathcal{X}}\left(
    x_t - \alpha_t \sum_{r = 0}^{n_a - 1}  g^a(y^a_{r, t}, \xi^r_t) - \alpha_t \sum_{r = 0}^{n_b - 1}  g^b(y^b_{r, t}, \xi^{n_a + r}_t) \right) - x_*.
\end{align*}

Since the sequence $\xi_t$ is drawn independently, using Assumption~\ref{ass:gradient} (a) one has
\[\mathbb{E}_{\xi_t}[g^a(y_{r, t}^a, \xi_{t}^r)] \;=\; \mathbb{E}_{\xi_t^0, \ldots, \xi_{t}^{r-1}}[\mathbb{E}_{\xi_t^r}[g^a(y_{r, t}^a, \xi_{t}^r)]] \;=\; \mathbb{E}_{\xi_t^0, \ldots, \xi_{t}^{r-1}}[\nabla f^a(y_{r, t}^a)] \;=\; \mathbb{E}_{\xi_t}[\nabla f^a(y_{r, t}^a)],\]
where the last equality holds due to the independence between $y_{r, t}^a$ and $\{\xi_t^{r}, \ldots, \xi_t^{n_a + n_b - 1}\}$. Similarly, we have $\mathbb{E}_{\xi_t}[g^b(y_{r, t}^b, \xi_{t}^{n_a + r})] = \mathbb{E}_{\xi_t}[\nabla f^b(y_{r, t}^b)]$.
Then, taking square norms and expectations over the random variables $\xi_t$ on both sides yields
\begin{equation}
\label{theorem2_eq1}
    \begin{split}
    \mathbb{E}_{\xi_t}[\|x_{t+1} - x_*\|^2]  
     \;\leq\; & \|x_t - x_*\|^2 + \alpha_t^2\mathbb{E}_{\xi_t}[\|\sum_{r = 0}^{n_a - 1} g^a(y^a_{r, t}, \xi^r_t) + \sum_{r = 0}^{n_b - 1} g^b(y^b_{r, t}, \xi^{n_a + r}_t)\|^2] \\
     & - \mathbb{E}_{\xi_t}[2\alpha_t(x_t - x_*)^\top (\sum_{r = 0}^{n_a - 1} \nabla f^a(y^a_{r, t}) + \sum_{r = 0}^{n_b - 1} \nabla f^b(y^b_{r, t}))],
    \end{split}
\end{equation}
which holds by the non-expansiveness property of the orthogonal projection operator.

We now claim, by applying a version of the Intermediate Value Theorem given in Proposition~\ref{prop_ivt}, that the last term of the right-hand side in~\eqref{theorem2_eq1} can be written as $- 2\alpha_t(x_t - x_*)^\top \mathbb{E}_{\xi_t}[n_a \nabla f^a(w_t^a) + n_b \nabla f^b(w_t^b)]$ for some $w_t^a$ and $w_t^b$. In fact, we apply Proposition~\ref{prop_ivt} to the real continuous function $\phi^i(y) = -2\alpha_t(x_t - x_*)^\top \nabla f^i(y)$, from which we then know that $w_t^i$ is a convex linear combination of a sequence of points $\{y_{r, t}^i\}_{r = 0}^{n_i - 1}$ for both $i \in \{a, b\}$. 

Using a combination of~(\ref{Mnabla}) and Assumption~\ref{ass:gradient}~(b), the bound for the second moment of the stochastic gradients is given by
\begin{equation} \label{SGbound}
\mathbb{E}_{\xi_t}[\|g^i(y^i_{r, t}, \xi_t^r)\|^2] \; \leq \; G + \bar{G}M_{\nabla}, \quad
\mathbb{E}_{\xi_t}[\|g^i(y^i_{r, t}, \xi_t^r)\|] \; \leq \; \sqrt{G + \bar{G}M_{\nabla}},
\end{equation} 
where the second bound results from applying Jensen's inequality to the first one.
Hence, the second term in the right-hand side of~\eqref{theorem2_eq1} can be bounded by
\begin{equation*}
\begin{split}
   & \mathbb{E}_{\xi_t}[\|\sum_{r = 0}^{n_a - 1} g^a(y^a_{r, t}, \xi^r_t) + \sum_{r = 0}^{n_b - 1}  g^b(y^b_{r, t}, \xi^{n_a + r}_t)\|^2] \\
   &  \;\leq\; 2\mathbb{E}_{\xi_t}[\|\sum_{r = 0}^{n_a - 1}  g^a(y^a_{r, t}, \xi^r_t)\|^2] + 2\mathbb{E}_{\xi_t}[\|\sum_{r = 0}^{n_b - 1}  g^b(y^b_{r, t}, \xi^{n_a+r}_t)\|^2] \\
& \;\leq\;  2n_a\sum_{r = 0}^{n_a - 1} \mathbb{E}_{\xi_t}[\| g^a(y^a_{r, t}, \xi^r_t)\|^2] + 2n_b\sum_{r = 0}^{n_b - 1} \mathbb{E}_{\xi_t}[\| g^b(y^b_{r, t}, \xi^{n_a + r}_t)\|^2] \\
& \;\leq\;  2(n_a + n_b)^2(G + \bar{G}M_{\nabla}).
\end{split}
\end{equation*}
We thus arrive at
\begin{equation}
    \label{theorem2_eq2}
    \begin{split}
    \mathbb{E}_{\xi_t}[\|x_{t+1} - x_*\|^2]
    \;\leq\; &  \|x_t - x_*\|^2 + 2\alpha_t^2(n_a + n_b)^2(G + \bar{G}M_{\nabla}) \\
     & - 2\alpha_t(x_t - x_*)^\top \mathbb{E}_{\xi_t}[n_a \nabla f^a(w_t^a) + n_b \nabla f^b(w_t^b)].
    \end{split}
\end{equation}

By adding and subtracting $2\alpha_t(x_t - x_*)^\top(n_a\nabla f^a(x_t) + n_b \nabla f^b(x_t))$ in the right-hand side of~\eqref{theorem2_eq2}, we further rewrite it as
\begin{equation}
\label{theorem2_eq22}
    \begin{split}
    \mathbb{E}_{\xi_t}[\|x_{t+1} - x_*\|^2]
    \;\leq\; &  \|x_t - x_*\|^2 + 2\alpha_t^2(n_a + n_b)^2(G + \bar{G}M_{\nabla}) \\
     & - 2\alpha_t(x_t - x_*)^\top (n_a \nabla f^a(x_t) + n_b \nabla f^b(x_t)) \\
     & + 2\alpha_t\|x_t - x_*\|\mathbb{E}_{\xi_t}[\|n_a \nabla f^a(w_t^a) - n_a \nabla f^a(x_t)\|] \\
     & + 2\alpha_t\|x_t - x_*\|\mathbb{E}_{\xi_t}[\|n_b \nabla f^b(w^b_t) - n_b \nabla f^b(x_t)\|]. 
    \end{split}
\end{equation}
Note that the last two terms are derived by the Cauchy–Schwarz and Jensen's inequalities.

\vspace{2ex}
\noindent {\bf Part II: Using strong convexity.} 
Selecting $\lambda_* = \lambda(n_a, n_b) = n_a/(n_a + n_b)$, by the strong convexity of the weighted-sum function, one has
\begin{equation*}
 \nabla_x S(x_t, \lambda_*)^\top(x_t - x_*)  \;\geq\; S(x_t, \lambda_*) - S(x_*, \lambda_*) + \frac{c}{2}\| x_t - x_*\|^2,
\end{equation*}
which is equivalent to
\begin{equation}
\label{theorm2_eq3}
 (x_t - x_*)^\top(n_a \nabla f^a(x_t) + n_b \nabla f^b(x_t))  \;\geq\; (n_a + n_b)(S(x_t, \lambda_*) - S(x_*, \lambda_*) + \frac{c}{2}\| x_t - x_*\|^2).
\end{equation}

From Assumption~\ref{ass:Lipschitz}, we obtain a bound for the last two terms of~\eqref{theorem2_eq22} in the form 
\begin{equation}
    \label{theorem2_eq4}
    2\alpha_t\|x_t - x_*\|\mathbb{E}_{\xi_t}[\|n_i \nabla f^i(w_t^i) - n_i \nabla f^i(x_t)\|] \;\leq\;  2\alpha_tLn_i\|x_t - x_*\|\mathbb{E}_{\xi_t}[\|x_t - w_t^i\|], \forall i \in \{a, b\}.
\end{equation}
According to Proposition~\ref{prop_ivt}, $w_t^i$ is a convex linear combination of a sequence of points $\{y_{r, t}^i\}_{r = 0}^{n_i -1}$ for $i \in \{a, b\}$. One can write $w_t^i = \sum_{r=0}^{n_i - 1} \beta_r y_{r, t}^i$ with $\beta_r \geq 0$, $r = 0, \ldots, n_i - 1$ , and $\sum_{r=0}^{n_i - 1} \beta_r = 1$. An explicit upper bound of $\|x_t - w_t^i\|$ can then be derived as follows
\begin{equation}
\label{part2_eq2}
    \begin{split} 
        \|x_t - \sum_{r=0}^{n_i - 1} \beta_r y_{r, t}^i\| & \;=\; \|\sum_{r=0}^{n_i - 1} \beta_r (x_t - y_{r, t}^i)\| \;\leq\; \sum_{r=0}^{n_i - 1} \beta_r \|x_t - y_{r, t}^i\|.
    \end{split}
\end{equation}
Using $y_{r, t}^i = y_{0, t}^i - \sum_{j = 0}^{r-1}\alpha_t g^i(y_{j, t}^i, \xi_{t}^j)$ and applying the triangle inequality, we have
\begin{align}
    \|x_t - y_{r, t}^a\| \;\leq\; \alpha_t \sum_{j = 0}^{r-1} \|g^a(y^a_{j, t}, \xi_{t}^j)\|,\label{part2_eq3}
\end{align}  
and
\begin{equation}
\label{part2_eq4}
\begin{split}
    \|x_t - y_{r, t}^b\| & \;=\; \|x_t - y_{n_a, t}^a + \sum_{j = 0}^{r-1}\alpha_t g^b(y_{j, t}^b, \xi_{t}^{n_a + j})\| \\
    & \;\leq\; \alpha_t \sum_{j = 0}^{n_a-1} \|g^a(y^a_{j, t}, \xi_{t}^j)\| +  \alpha_t \sum_{j = 0}^{r-1} \|g^b(y^b_{j, t}, \xi_{t}^{n_a + j})\|.
\end{split}
\end{equation}
Plugging~\eqref{part2_eq3} into~\eqref{part2_eq2} with $i = a$ results in (note that $y^a_{0,t}=x_t$)
\begin{equation}\label{part2_eq5}
\begin{split}
       \|x_t - w_t^a\| & \;\leq\; \alpha_t \sum_{r=1}^{n_a - 1} \beta_r \sum_{j = 0}^{r-1} \|g^a(y^a_{j, t}, \xi_{t}^j)\| \\ & \;  = \;  \alpha_t 
       \sum_{j=0}^{n_a - 1} \|g^a(y^a_{j, t}, \xi_{t}^j)\|
      \sum_{r=j+1}^{n_a-1} \beta_r
       \;\leq \; \alpha_t  \sum_{j=0}^{n_a - 1} \|g^a(y^a_{j,t}, \xi_{t}^{j})\|.
\end{split}
\end{equation}

Applying expectations on both sides of~\eqref{part2_eq5}, followed by the bound~(\ref{SGbound}) on the second moment of the stochastic gradients,
\begin{equation} \label{part2_eq6}
\mathbb{E}_{\xi_t}[\|x_t - w_t^a\|] \; \leq \; \alpha_t n_a\sqrt{G + \bar{G}M_{\nabla}}.
\end{equation}
Similarly, merging~\eqref{part2_eq4} into~\eqref{part2_eq2} with $i = b$ leads to 
\begin{equation}
\label{part2_eq6.5}
        \|x_t - w_t^b\| \;\leq \; \alpha_t \sum_{j=0}^{n_a - 1} \|g^a(y^a_{j,t}, \xi_{t}^{j})\| + \alpha_t  \sum_{j=0}^{n_b - 1} \|g^b(y^b_{j,t}, \xi_{t}^{n_a+j})\|,
\end{equation}
and applying expectations
\begin{equation}
\label{part2_eq7}
        \mathbb{E}_{\xi_t}[\|x_t - w_t^b\|] \;\leq \; \alpha_t (n_a+n_b) \sqrt{G + \bar{G}M_{\nabla}}.
\end{equation}
Finally, combining~\eqref{theorem2_eq4}, \eqref{part2_eq6}, and~\eqref{part2_eq7} yields
\begin{equation}
\label{part2_eq8}
    \begin{split}
        2\alpha_t\|x_t - x_*\|\sum_{i \in \{a, b\}} \mathbb{E}_{\xi_t}[\|n_i \nabla f^i(w_t^i) - n_i \nabla f^i(x_t)\|] \;\leq\; 2\alpha_t^2L\Theta (n_a + n_b)^2\sqrt{G + \bar{G}M_{\nabla}}.
    \end{split}
\end{equation}

\vspace{2ex}
\noindent {\bf Part III: Bound on the optimality gap in terms of weighted-sum function.}
Applying inequalities \eqref{theorm2_eq3} and \eqref{part2_eq8} to \eqref{theorem2_eq22} leads to 
\begin{equation}
\label{smooth_sc_part3_eq1}
    \begin{split}
     \mathbb{E}_{\xi_t}[\|x_{t+1} - x_*\|^2]  \;\leq\; &  (1-\alpha_t(n_a + n_b)c) \|x_t - x_*\|^2 -  2\alpha_t(n_a + n_b)(S(x_t, \lambda_*) - S(x_*, \lambda_*)) \\
     & + \alpha^2_t M,
    \end{split}
\end{equation}
where we let $M = 2(n_a + n_b)^2 (G + \bar{G}M_{\nabla} +  L\Theta\sqrt{G + \bar{G}M_{\nabla}})$. Plugging in $\alpha_t = \frac{2}{c(n_a+n_b)(t+1)}$ and rearranging the last inequality result in 
\begin{equation*}
    \begin{split}
        S(x_t, \lambda_*) - S(x_*, \lambda_*)  \;\leq\; & \frac{(1-\alpha_t(n_a + n_b)c) \|x_t - x_*\|^2 - \mathbb{E}_{\xi_t}[\|x_{t+1} - x_*\|^2] + \alpha^2_t M}{2\alpha_t(n_a + n_b)} \\
        \;\leq\; & \frac{c(t-1)}{4}\|x_t - x_*\|^2  - \frac{c(t+1)}{4} \mathbb{E}_{\xi_t}[\|x_{t+1} - x_*\|^2] + \frac{\tilde{M}}{c(t+1)},
    \end{split}
\end{equation*}
where $\tilde{M} = \frac{M}{(n_a + n_b)^2}$. By taking total expectation over $\{\xi_t\}$, multiplying both sides by $t$, and summing over $t = 1, \ldots, T$, one obtains
\begin{equation*}
    \begin{split}
       \sum_{t=1}^T 
       t(\mathbb{E}[S(x_t, \lambda_*)] - S(x_*, \lambda_*)) \;\leq\; & \sum_{t=1}^T \left ( \frac{ct(t-1)}{4}\mathbb{E}[\|x_t - x_*\|^2]  - \frac{ct(t+1)}{4} \mathbb{E}[\|x_{t+1} - x_*\|^2]\right) \\
       & + \sum_{t=1}^T \frac{\tilde{M}t}{c(t+1)} \\
       \;\leq\; & -\frac{cT(T+1)}{4} \mathbb{E}[\|x_{T+1} - x_*\|^2] + \sum_{t=1}^T \frac{\tilde{M}t}{c(t+1)}  \;\leq\; \frac{T}{c} \tilde{M}. 
    \end{split}
\end{equation*}
Dividing both sides of the last inequality by $\sum_{t=1}^T t$ yields
\begin{equation*}
    \min_{t = 1, \ldots,T} \mathbb{E}[S(x_t, \lambda_*)] - S(x_*, \lambda_*) \;\leq\; \frac{2}{c(T + 1)} \tilde{M},   
\end{equation*}
which concludes the proof.
\end{proof}

 The key of the proof was to bound the deviation from $w^i_t$ to $x_t$ using the step size $\alpha_t$. It is in fact the use of the decaying step size that compensates for the error $\mathcal{O}(\alpha_t^2)$ generated when bundling the gradients using the Intermediate Value Theorem. Note that the above theorem is still guaranteed when one relaxes Assumption~\ref{ass:strongconv} to the case where both objectives are convex and one of them is strongly convex, as~\eqref{theorm2_eq3} still holds in such a case. A rate in terms of the iterates can be also derived.
 
\begin{corollary}
\label{th:smooth_convex_iterate}
 Let Assumptions~\ref{ass:limit_pts}-\ref{ass:gradient} hold and $x_*$ be the unique minimizer of the weighted function $S(\cdot, \lambda_*)$ in~$\mathcal{X}$, where $\lambda_* = \lambda(n_a, n_b) = n_a/(n_a + n_b)$. Choosing a diminishing step size sequence
$\alpha_t = \gamma /t$ where $\gamma > \frac{1}{2(n_a + n_b)c}$, the sequence of iterates generated by the SA2GD algorithm satisfies
\begin{equation*}
    \mathbb{E}[\|x_T - x_*\|^2] \;\leq\; \frac{\max\{2\gamma^2M(2c (n_a + n_b) \gamma -1)^{-1}, \|x_0-x_*\|^2\}}{T}.   
\end{equation*}
 where $M = 2(n_a + n_b)^2 (G + \bar{G}M_{\nabla} + L\Theta\sqrt{G + \bar{G}M_{\nabla}})$.
\end{corollary} 

\begin{proof} 
 Since the weighted-sum function $S(\cdot, \lambda_*)$ is strongly convex, one also has 
$$
(\nabla S(x, \lambda) - \nabla_x S(\bar{x}, \lambda) )^\top (x - \bar{x}) \;\geq\; c\|x - \bar{x}\|^2, \quad \forall (x, \bar{x}) \in \mathcal{X} \times \mathcal{X}.
$$
Letting $x = x_t$, $\bar{x} = x_*$, and $\lambda = \lambda_k$ in the above inequality leads to 
$$
\nabla S(x_t, \lambda_*)^\top (x_t - x_*) \;\geq\; c\|x_t - x_*\|^2,
$$
where we used the fact that $\nabla S(x_*, \lambda_* )^\top (x_t - x_*)
= (P_{\mathcal{X}} \nabla S(x_*, \lambda_* ))^\top (x_t - x_*) = 0$.
Hence,
\begin{equation}
\label{smooth_sc_iter_eq1}
(x_t - x_*)^\top(n_a \nabla f^a(x_t) + n_b \nabla f^b(x_t))  \;\geq\; c (n_a+n_b) \|x_t - x_*\|^2.
\end{equation}

 Plugging~\eqref{part2_eq8} and~\eqref{smooth_sc_iter_eq1} in~\eqref{theorem2_eq22} results in
$$
\mathbb{E}[\|x_{t+1}- x_*\|^2] \;\leq\; (1 -  2c(n_a + n_b)\alpha_t)\mathbb{E}[\| x_{t}- x_*\|^2] + \alpha_t^2 M.
$$
Using $\alpha_t = \gamma/t$ with $\gamma > 1/(2c(n_a + n_b))$ and
the induction argument in~\cite[Eq.~(2.9) and (2.10)]{ANemirovski_2009} lead us to the result.
\end{proof}  
 
\subsection{The non-smooth and strongly convex case}
\label{smooth_agd}

In this section, we analyze the convergence rates for stochastic alternating bi-objective descent (SA2GD) when subgradients are used in Algorithm~\ref{alg1_SA2GD}. Instead of using gradients, the notation $g^i(x, \xi), \forall i \in \{a, b\}$, denotes now the individual stochastic subgradient generated using a random variable $\xi$. Let $\partial f^i(x)$ be the subdifferential at $x$. Next, we describe the assumptions under which SA2GD will be analyzed (in addition to Assumption~\ref{ass:limit_pts}). We formalize first a classical assumption of Lipschitz continuity of both objective functions, which is often satisfied in practice. 

\begin{assumption}
\label{ass_LipCont_func}
\textbf{(Lipschitz continuous functions)} Both objective functions are Lipschitz continuous with Lipschitz constants $\hat{L}_i > 0$, $i \in \{a, b\}$, i.e.,
\begin{equation*}
|f^i(\bar{x}) - f^i(x)| \;\leq\; \hat{L}_i\|\bar{x} - x\|, \quad \forall (x,\bar{x}) \in \mathcal{X} \times \mathcal{X}.
\end{equation*}
\end{assumption}

We also impose the following two classical assumptions about stochastic subgradients. 

\begin{assumption}
\label{ass_subgradient}
For both objective functions $i \in \{a, b\}$, and all iterates $t \in \mathbb{N}$, the stochastic subgradients~$g^i(x_t, \xi_t)$ satisfy the following:
\begin{enumerate}
    \item[(a)]\textbf{(Unbiasedness)} 
    $\mathbb{E}_{\xi_t}[g^i(x_t, \xi_t)] \;\in\; \partial f^i(x_t)$. 
    \item[(b)](\textbf{Boundness}) There exist positive constants $\tilde{L}_i > 0$ such that $\mathbb{E}_{\xi_t}[\|g^i(x_t, \xi_t)\|^2] \leq \tilde{L}^2_i$. 
\end{enumerate}
\end{assumption}

The above assumptions are commonly used ones in classical stochastic subgradient methods~\cite[Section 8.3]{ABeck_2017}. For conciseness, we denote $\hat{L} = \max\{\hat{L}_a, \hat{L}_b\}$ and $\tilde{L} = \max\{\tilde{L}_a, \tilde{L}_b\}$. 

In addition to the above assumptions, we impose strong convexity in both objective functions.

\begin{assumption}
\label{ass_strong_convexity}
\textbf{(Strong convexity in the non-smooth case)} For both objective functions $i \in \{a, b\}$, there exists a scalar $\hat{c}_i > 0$ such that
\begin{equation*}
f^i(\bar{x}) \;\geq\; f^i(x) + g^i(x)^\top(\bar{x} - x) + \frac{\hat{c}_i}{2}\|\bar{x} - x\|^2, \quad \forall (x,\bar{x}) \in \mathcal{X} \times \mathcal{X},
\end{equation*}
for all subgradients $g^i(x) \in \partial f^i(x)$.
\end{assumption}

Based on the above two assumptions, the weighted-sum function $S(x, \lambda) = \lambda f^a(x) + (1-\lambda)f^b(x), \forall \lambda \in [0, 1]$, is also Lipschitz continuous with constant $\hat{L}$ and strongly convex with constant $\hat{c} = \min\{\hat{c}_a, \hat{c}_b\}$. Now, we are ready to show a similar convergence rate result as Theorem~\ref{th:smooth_strongcov}. Readers are referred to~\cite[Theorems~8.31~(a) and 8.37~(a)]{ABeck_2017} for the single-objective counterparts.

\begin{theorem} (\textbf{Sublinear convergence rate of SA2GD in the non-smooth and strongly convex case})
\label{th:nonsmooth_sc_convex}
Let Assumptions~\ref{ass:limit_pts} and~\ref{ass_LipCont_func}--\ref{ass_strong_convexity} hold and $x_*$ be the unique minimizer of the weighted function $S(\cdot, \lambda_*)$ in~$\mathcal{X}$, where $\lambda_* = \lambda(n_a, n_b) = n_a/(n_a + n_b)$. Choosing a diminishing step size sequence $\alpha_t = \frac{2}{\hat{c}(t+1)(n_a + n_b)}$, the sequence of iterates generated by the SA2GD algorithm satisfies
\begin{equation*}
    \min_{t = 1, \ldots,T} \mathbb{E}[S(x_t, \lambda_*)] - S(x_*, \lambda_*) \;\leq\; \frac{4}{\hat{c}(T + 1)} \left( 2\tilde{L}^2 +\hat{L}\tilde{L} + \hat{c}\Theta\tilde{L} \right).   
\end{equation*}
\end{theorem}

\begin{proof}
Similar to the proof of Theorem~\ref{th:smooth_strongcov}, the proof is divided in three parts. In the first part, one obtains an upper bound on the norm of the iterates, $\mathbb{E}_{\xi_t}[\|x_{t+1} - x_*\|^2]$. Strong convexity is applied in the second part. We skip the third part as it is exactly the same as part III of Theorem~\ref{th:smooth_strongcov}. \vspace{1ex}

\noindent {\bf Part I: Bound on the iterates error.} 
At any iteration $t$, the sequence of stochastic subgradients is computed from drawing the sequence of random variables $\xi_t = \{\xi^0_t, \ldots, \xi^{n_a + n_b - 1}_t \}$. Replacing the gradients in~\eqref{theorem2_eq1} by the corresponding subgradients, one has
\begin{equation}
\label{theorem_sa2sg_sc_eq1}
\begin{split}
    \mathbb{E}_{\xi_t}[\|x_{t+1} - x_*\|^2]
    \;\leq\; & \|x_t - x_*\|^2 + \alpha_t^2\mathbb{E}_{\xi_t}[\|\sum_{r = 0}^{n_a - 1} g^a(y^a_{r, t}, \xi^r_t) + \sum_{r = 0}^{n_b - 1} g^b(y^b_{r, t}, \xi^{n_a + r}_t)\|^2] \\
    & - 2\alpha_t\mathbb{E}_{\xi_t}[(x_t - x_*)^\top(\sum_{r = 0}^{n_a - 1} (f^a)'(y^a_{r, t}) + \sum_{r = 0}^{n_b - 1} (f^b)'(y^b_{r, t}))],
\end{split}
\end{equation}
where $(f^a)'(x) \in \partial f^a(x)$ and $(f^b)'(x) \in \partial f^b(x)$ denote certain deterministic subgradients at $x$ for the two objectives respectively. 

The upper bound for the second moment of the sequence of stochastic subgradients \linebreak $\{g^a(y^a_{r, t}, \xi_t^r)\}_{r = 0}^{n_a - 1}$ and $\{g^b(y^b_{r, t}, \xi_t^{n_a + r})\}_{r = 0}^{n_b - 1}$ at each iteration $t$ can be derived as follows.
\begin{fleqn}[\parindent]
\begin{equation}
 \label{theorem_sa2sg_sc_eq2}
\begin{split}
  & \mathbb{E}_{\xi_t}[\|\sum_{r = 0}^{n_a - 1} g^a(y^a_{r, t}, \xi^r_t) + \sum_{r = 0}^{n_b - 1} g^b(y^b_{r, t}, \xi^{n_a + r}_t)\|^2] \\
  & \;\leq\;  2\mathbb{E}_{\xi_t}[\|\sum_{r = 0}^{n_a - 1} g^a(y^a_{r, t}, \xi^r_t)\|^2] + 2\mathbb{E}_{\xi_t}[\|\sum_{r = 0}^{n_b - 1} g^b(y^b_{r, t}, \xi^{n_a+r}_t)\|^2] \\
    & \;\leq\;  2n_a\sum_{r = 0}^{n_a - 1} \mathbb{E}_{\xi_t}[\|g^a(y^a_{r, t}, \xi^r_t)\|^2] + 2n_b\sum_{r = 0}^{n_b - 1} \mathbb{E}_{\xi_t}[\|g^b(y^b_{r, t}, \xi^{n_a + r}_t)\|^2]\\ 
   &  \;\leq\; 2(n_a + n_b)^2\tilde{L}^2,
\end{split}
\end{equation}
\end{fleqn}
where $\tilde{L} = \max(\tilde{L}_a, \tilde{L}_b)$. 

\vspace{2ex}
\noindent {\bf Part II: Using strong convexity and Intermediate Value Theorem.} 
In order to get a bound for the term $-2\alpha_t\mathbb{E}_{\xi_t}[(x_t - x_*)^\top\sum_{r = 0}^{n_a - 1} (f^a)'(y^a_{r, t})] -2\alpha_t\mathbb{E}_{\xi_t}[(x_t - x_*)^\top\sum_{r = 0}^{n_b - 1} (f^b)'(y^b_{r, t})]$, we will apply Assumption~\ref{ass_strong_convexity} multiple times. For simplicity, the expectation symbol is temporarily ignored,
\begin{equation}
\label{theorem_sa2sg_sc_eq3}
\begin{split}
    & -2\alpha_t (x_t - x_*)^\top \sum_{r = 0}^{n_a - 1} (f^a)'(y^a_{r, t})  \\ 
    \;=\;   & -2\alpha_t(x_t - x_*)^\top(f^a)'(x_t)  \\
    & - 2\alpha_t(y^a_{1, t} - x_*)^\top(f^a)'(y^a_{1, t})   -2\alpha_t(x_t - y^a_{1, t})^\top(f^a)'(y^a_{1, t}) \\
    & \cdots \\
    &  - 2\alpha_t(y^a_{n_a - 1, t} - x_*)^\top(f^a)'(y^a_{n_a - 1, t}) -2\alpha_t(x_t - y^a_{n_a - 1, t})^\top(f^a)'(y^a_{n_a - 1, t}) \\
    \;\leq\; & -2\alpha_t\sum_{r= 0}^{n_a - 1} (f^a(y_{r, t}^a)  - f^a(x_*) + \frac{\hat{c}}{2}\|y^a_{r, t} - x_*\|^2)  - 2\alpha_t \sum_{r=0}^{n_a - 1}(x_t - y^a_{r, t})^\top(f^a)'(y^a_{r, t}) \\
    \;\leq\; & -2\alpha_t(n_af^a(w_t^a) - n_a f^a(x_*)) - \hat{c}\alpha_tn_a\|w_t^a - x_*\|^2  + 2\alpha_t \sum_{r=0}^{n_a - 1}\|x_t - y^a_{r, t}\|\tilde{L}_a,
\end{split}
\end{equation}
where we applied Assumption~\ref{ass_subgradient}~(a) and
Assumption~\ref{ass_strong_convexity} to get the first inequality. The last inequality holds by applying Proposition~\ref{prop_ivt} to the quadratic function $\psi^a(x) = f^a(x) - f^a(x_*) + \frac{\hat{c}}{2}\|x - x_*\|^2$, the Cauchy–Schwarz inequality, and the combination of Assumption~\ref{ass_subgradient}~(b) and the Jensen's inequality $\|\mathbb{E}_{\xi_t}[g^i(x_t, \xi_t)]\|^2 \leq \mathbb{E}_{\xi_t}[\|g^i(x_t, \xi_t)\|^2]$. Here, $w_t^a$ is a point in the convex hull of $\{y^a_{r, t}\}_{r=0}^{n_a - 1}$. 

The second term $\|w_t^a - x_*\|^2$ on the right hand side can be handled using the triangle inequality, i.e.,
\begin{equation*}
    \begin{split}
        \|w_t^a - x_*\|^2 &  \;\geq\; (\|w_t^a - x_t\| - \|x_t - x_*\|)^2 \\
        & \;=\; \|x_t - w_t^a\|^2 + \|x_t - x_*\|^2 - 2\|x_t - w_t^a\|\|x_t - x_*\| \\
        & \;\geq\; \|x_t - x_*\|^2 - 2\|x_t - w_t^a\|\|x_t - x_*\|.
    \end{split}
\end{equation*}

Applying the above inequality to~\eqref{theorem_sa2sg_sc_eq3} and using Assumption~\ref{ass:limit_pts} result in
\begin{equation}
\label{theorem_sa2sg_sc_eq4}
    \begin{split}
        -2\alpha_t (x_t - x_*)^\top \sum_{r = 0}^{n_a - 1} (f^a)'(y^a_{r, t})  \;\leq\; & -2\alpha_t(n_af^a(w_t^a) - n_a f^a(x_*)) - \hat{c}\alpha_tn_a\|x_t - x_*\|^2 \\
        & + 2\hat{c}\Theta\alpha_tn_a\|x_t - w_t^a\| + 2\alpha_t \sum_{r=0}^{n_a - 1}\|x_t - y^a_{r, t}\|\tilde{L}_a.
    \end{split}
\end{equation}
Similarly, we can obtain
\begin{equation}
\label{theorem_sa2sg_sc_eq5}
\begin{split}
    -2\alpha_t (x_t - x_*)^\top \sum_{r = 0}^{n_b - 1} (f^b)'(y^b_{r, t}) \;\leq\; & -2\alpha_t(n_bf^b(w_t^b) - n_bf^b(x_*)) - \hat{c}\alpha_tn_b\|x_t - x_*\|^2 \\
    & + 2\hat{c}\Theta\alpha_tn_b\|x_t - w_t^b\| + 2\alpha_t \sum_{r=0}^{n_b - 1}\|x_t - y^b_{r, t}\|\tilde{L}_b,
\end{split}
\end{equation}
where $w_t^b$ is a point in the convex hull of $\{y^b_{r, t}\}_{r=0}^{n_b - 1}$. \vspace{1ex}

Plugging the inequalities~\eqref{theorem_sa2sg_sc_eq2}, \eqref{theorem_sa2sg_sc_eq4}-\eqref{theorem_sa2sg_sc_eq5} back in~\eqref{theorem_sa2sg_sc_eq1} and adding and subtracting $2\alpha_t(n_a f^a(x_t) + n_b f^b(x_t))$ on the right-hand side yield
\begin{equation}
\label{theorem_sa2sg_sc_eq6}
    \begin{split}
    \mathbb{E}_{\xi_t}[\|x_{t+1} - x_*\|^2]
    \;\leq\; &  (1-\hat{c}\alpha_t(n_a + n_b))\|x_t - x_*\|^2 + 2\alpha_t^2(n_a + n_b)^2\tilde{L}^2 \\
    & - 2\alpha_t(n_a f^a(x_t) + n_b f^b(x_t) - (n_af^a(x_*) + n_bf^b(x_*))) \\
    & - 2\alpha_t\mathbb{E}_{\xi_t}[n_af^a(w_t^a) - n_af^a(x_t) + n_bf^b(w^b_t) - n_bf^b(x_t)] \\
    & + 2\hat{c}\Theta\alpha_t (n_a\mathbb{E}_{\xi_t}[\|x_t - w_t^a\|] + n_b\mathbb{E}_{\xi_t}[\|x_t - w_t^b\|]) \\
    & + 2\alpha_t \tilde{L}(\sum_{r=0}^{n_a - 1}\mathbb{E}_{\xi_t}[\|x_t - y^a_{r, t} \|] + \sum_{r=0}^{n_b - 1}\mathbb{E}_{\xi_t}[\|x_t - y^b_{r, t}\|]) \\
    \;\leq\; &  (1-\hat{c}\alpha_t(n_a + n_b))\|x_t - x_*\|^2 + 2\alpha_t^2(n_a + n_b)^2\tilde{L}^2 \\
    & - 2\alpha_t(n_a f^a(x_t) + n_b f^b(x_t) - (n_af^a(x_*) + n_bf^b(x_*))) \\
    & + 2\hat{L}\alpha_t(n_a\mathbb{E}_{\xi_t}[\|x_t - w_t^a\|] + n_b\mathbb{E}_{\xi_t}[\|x_t - w^b_t\|]) \\
    & + 2\hat{c}\Theta\alpha_t (n_a\mathbb{E}_{\xi_t}[\|x_t - w_t^a\|] + n_b\mathbb{E}_{\xi_t}[\|x_t - w_t^b\|]) \\
    & + 2\alpha_t \tilde{L}(\sum_{r=0}^{n_a - 1}\mathbb{E}_{\xi_t}[\|x_t - y^a_{r, t} \|] + \sum_{r=0}^{n_b - 1}\mathbb{E}_{\xi_t}[\|x_t - y^b_{r, t}\|]),\\
    \end{split}
\end{equation}
where, in the second inequality, we applied Assumption~\ref{ass_LipCont_func} to the fourth term.

We now recall the bounds~\eqref{part2_eq3}--\eqref{part2_eq5} and \eqref{part2_eq6.5} on the distance from~$y^a_{r,t}$, $y^b_{r,t}$, $w_t^a$, and $w_t^b$ to~$x_t$. Applying expectation to these four bounds, followed by Assumption~\ref{ass_subgradient}~(b), 
gives us
\begin{equation} \label{part2_eqC}
\begin{split}
\mathbb{E}_{\xi_t}[\|x_t - y^a_{r,t}\|] \; \leq & \;\; \alpha_t n_a \tilde{L} \\
\mathbb{E}_{\xi_t}[\|x_t - y^b_{r,t}\|] \; \leq & \;\; \alpha_t (n_a+n_b) \tilde{L} \\
\mathbb{E}_{\xi_t}[\|x_t - w_t^a\|] \; \leq & \;\; \alpha_t n_a \tilde{L} \\
\mathbb{E}_{\xi_t}[\|x_t - w_t^b\|] \;\leq & \;\; \alpha_t (n_a+n_b) \tilde{L}.
\end{split}
\end{equation}
Finally, plugging~(\ref{part2_eqC}) into~(\ref{theorem_sa2sg_sc_eq6}) yields
\begin{equation*} \label{theorem_sa2sg_sc_eq12}
    \begin{split}
    \mathbb{E}_{\xi_t}[\|x_{t+1} - x_*\|^2]
    \;\leq\; &  (1-\hat{c}\alpha_t(n_a + n_b))\|x_t - x_*\|^2 + \alpha_t^2M \\
    & - 2\alpha_t(n_a f^a(x_t) + n_b f^b(x_t) - (n_af^a(x_*) + n_bf^b(x_*))),
    \end{split}
\end{equation*}
where $M = (4\tilde{L}^2 +2\hat{L}\tilde{L} + 2\hat{c}\Theta\tilde{L})(n_a + n_b)^2$. 
Now, we arrive at exactly the same form of upper bound as~\eqref{smooth_sc_part3_eq1} in the smooth case. Choosing the step size sequence $\alpha_t = \frac{2}{\hat{c}(t+1)(n_a + n_b)}$, the proof is then completed by the same derivation as the Part III of the proof for Theorem~\ref{th:smooth_strongcov}.
\end{proof}

In the above proof, since we could not use Lipschitz continuity of subgradients, we applied the Intermediate Value Theorem after (instead of before) using the strong convexity assumption. The resulting quadratic term $\|w_t^i - x_*\|^2$ from strong convexity was then merged into the $\|x_t - x_*\|^2$ term and the other alike $\mathcal{O}(\alpha_t^2)$ terms, converting the bound on the iterate error to what we had seen before in the smooth and strongly convex case. 

\subsection{The convex case}

If we replace strong convexity (Assumption~\ref{ass_strong_convexity})  by (simple) convexity (see below), we will observe a degradation on the convergence rate. Corollary~\ref{th:nonsmooth_convex} can be seen as a counterpart of the single objective cases~\cite[Theorems~8.30~(b) and 8.35~(b)]{ABeck_2017}.

\begin{assumption}
\label{ass_convexity}
\textbf{(Convexity in the non-smooth case)} For both objective functions $i \in \{a, b\}$, we have
\begin{equation*}
f^i(\bar{x}) \;\geq\; f^i(x) + g^i(x)^\top(\bar{x} - x), \quad \forall (x,\bar{x}) \in \mathcal{X} \times \mathcal{X},
\end{equation*}
for all subgradients $g^i(x) \in \partial f^i(x)$.
\end{assumption}

\begin{corollary} (\textbf{Sublinear convergence rate of SA2GD in the non-smooth and convex case})
\label{th:nonsmooth_convex}
Let Assumptions~\ref{ass:limit_pts}, \ref{ass_LipCont_func}--\ref{ass_subgradient}, and~\ref{ass_convexity} hold and $x_*$ be a minimizer of the weighted function $S(\cdot, \lambda_*)$ in~$\mathcal{X}$, where $\lambda_* = \lambda(n_a, n_b) = n_a/(n_a + n_b)$. Choosing a diminishing step size sequence
$\alpha_t = \frac{\bar{\alpha}}{\sqrt{t}(n_a + n_b)}$, where $\bar{\alpha}$ is any positive constant, the sequence of iterates generated by the SA2GD algorithm satisfies
\begin{equation*}
    \min_{t = 1, \ldots, T} \mathbb{E}[S({x_t, \lambda}_*)] - \mathbb{E}[S(x_*, \lambda_*)] \;\leq\; \frac{\frac{\Theta^2}{2\bar{\alpha}} + 4\bar{\alpha}\tilde{L}^2 +  2\bar{\alpha}\tilde{L}\hat{L}}{\sqrt{T}}.
\end{equation*}
\end{corollary}


\begin{proof}
The proof differs from the proof of Theorem~\ref{th:nonsmooth_sc_convex} in Parts II and III. When applying convexity to the second inequality of~\eqref{theorem_sa2sg_sc_eq3}, the term $\hat{c}\alpha_tn_a\|w^a_t - x_*\|^2$ does not appear on the right-hand side anymore, and we derive the bound of the iterate error
\begin{equation}
\label{theorem_sa2sg_cov_eq1}
    \begin{split}
    \mathbb{E}_{\xi_t}[\|x_{t+1} - x_*\|^2]
    \;\leq\; &  \|x_t - x_*\|^2 + 2\alpha_t^2(n_a + n_b)^2\tilde{L}^2 \\
    & - 2\alpha_t(n_a f^a(x_t) + n_b f^b(x_t) - (n_af^a(x_*) + n_bf^b(x_*))) \\
    & + 2\alpha_t\hat{L}\mathbb{E}_{\xi_t}[n_a\|x_t - w_t^a\| + n_b\|x_t - w^b_t\|] \\
    & + 2\alpha_t \tilde{L}(\sum_{r=0}^{n_a - 1}\mathbb{E}_{\xi_t}[\|x_t - y^a_{r, t} \|] + \sum_{r=0}^{n_b - 1}\mathbb{E}_{\xi_t}[\|x_t - y^b_{r, t}\|]).\\
    \end{split}
\end{equation}
Plugging~\eqref{part2_eqC} into~\eqref{theorem_sa2sg_cov_eq1} yields
\begin{equation*}
    \begin{split}
    \mathbb{E}_{\xi_t}[\|x_{t+1} - x_*\|^2]
    \;\leq\; &  \|x_t - x_*\|^2 + \alpha_t^2\hat{M} \\
    & - 2\alpha_t(n_a f^a(x_t) + n_b f^b(x_t) - (n_af^a(x_*) + n_bf^b(x_*))),
    \end{split}
\end{equation*}
where $\hat{M} = (4\tilde{L}^2+2\tilde{L}\hat{L} )(n_a + n_b)^2$. The proof is completed by applying the same standard arguments as in Theorem~5.3 of~\cite{SLiu_LNVicente_2019}.

\end{proof}

The non-smooth and convex case sheds light on the convergence behavior of the SAfairKM algorithm for the original non-convex and non-smooth bi-objective optimization problem~\cite{SLiu_LNVicente_2021}. 

For completeness, we conclude this section by the following corollary stating the convergence rate for the smooth and convex case without proof. The proof is also derived by getting rid of the quadratic term $-c\alpha_t(n_a + n_b)\|x_t - x_*\|^2$ in~\eqref{smooth_sc_part3_eq1} and then applying the standard arguments as Theorem~5.3 in~\cite{SLiu_LNVicente_2019}.

\begin{assumption}
\label{ass_smooth_convexity}
\textbf{(Convexity in the smooth case)} For both objective functions $i \in \{a, b\}$, we have
\begin{equation*}
f^i(\bar{x}) \;\geq\; f^i(x) + \nabla f^i(x)^\top(\bar{x} - x), \quad \forall (x,\bar{x}) \in \mathcal{X} \times \mathcal{X}.
\end{equation*}
\end{assumption}

\begin{corollary} (\textbf{Sublinear convergence rate of SA2GD in the smooth and convex case})
\label{th:smooth_convex}
Let Assumptions~\ref{ass:limit_pts}--\ref{ass:Lipschitz}, \ref{ass:gradient}, and~\ref{ass_smooth_convexity} hold and $x_*$ be a minimizer of the weighted function $S(\cdot, \lambda_*)$ in~$\mathcal{X}$, where $\lambda_* = \lambda(n_a, n_b) = n_a/(n_a + n_b)$. Choosing a diminishing step size sequence
$\alpha_t = \frac{\bar{\alpha}}{\sqrt{t}(n_a + n_b)}$, where $\bar{\alpha}$ is any positive constant, the sequence of iterates generated by the SA2GD algorithm satisfies
\begin{equation*}
    \min_{t = 1, \ldots, T} \mathbb{E}[S({x_t, \lambda}_*)] - \mathbb{E}[S(x_*, \lambda_*)] \;\leq\; \frac{\frac{\Theta^2}{2\bar{\alpha}} + 2\bar{\alpha} \hat{G}^2 + 2 \bar{\alpha}L\Theta \hat{G}}{\sqrt{T}}.
\end{equation*} 
\end{corollary}

As a final remark, the given upper bounds in Theorems~\ref{th:smooth_strongcov}-\ref{th:nonsmooth_sc_convex} and Corollaries~\ref{th:nonsmooth_convex}-\ref{th:smooth_convex} also hold for the aggregated iterate $\bar{x}_T = \textstyle \sum_{t=1}^Ttx_t/\sum_{t=1}^Tt$ in the strongly convex case and $\bar{x}_T = \textstyle\sum_{t=1}^Tx_t/T$ in the convex case. Such claims follow from the following Jensen's inequalities
$$S\left(\tfrac{\sum_{t=1}^Ttx_t}{\sum_{t=1}^Tt}, \lambda_*\right) \;\leq\; \tfrac{\sum_{t=1}^TtS(x_t, \lambda_*)}{\sum_{t=1}^T t} \quad \text{ and } \quad S\left(\tfrac{\sum_{t=1}^Tx_t}{T}, \lambda_* \right) \;\leq\; \tfrac{\sum_{t=1}^TS(x_t, \lambda_*)}{T}.$$

\section{A numerical experiment to illustrate the scalarization by optimization effort}

We now present an illustration of the behavior of the SA2GD method in what regards its ability to determine the whole Pareto front when applied multiple times over a discretization of the optimization effort scalars~$n_a$ and~$n_b$.
For simplicity we have applied SA2GD to deterministic problems, thus using full batch gradients.
Four deterministic bi-objective problems were selected from~\cite{ALCustodio_etal_2011} involving only simple bound constraints. 
We will compare the scalarization effort of SA2GD to the one of the weighted-sum method, where each scalarized problem
$\min_{x \in \mathcal{X}} \lambda f^a(x) + (1-\lambda) f^b(x)$, $\lambda \in (0,1)$, is solved by the same gradient descent methodology. For both approaches, we use a fixed step size $10^{-3}$ at each iteration. The starting point for each run is randomly generated within the feasible region.

Figure~\ref{Pareto_fronts} illustrates the approximated Pareto fronts obtained from both the SA2GD algorithm and the weight-sum algorithm for the selected problems. SA2GD was ran $201$ times with $n_a  + n_b = 200$ and $n_a \in \{0, ..., 200\}$. Weighted-sum was also run $201$ times with $\lambda \in \{0,1/200,\ldots,1\}$.
In both cases, we stopped each algorithmic run after $300$~iterations. 
Regardless of the shape of Pareto front (convex, concave, or disconnected), one observes that SA2GD performs similarly to the weight-sum approach in terms of capturing well-spread Pareto fronts from scalarization.

\begin{figure}[h]
   \centering
   \subfloat[][Problem MOP1.]{\includegraphics[width=.45\textwidth]{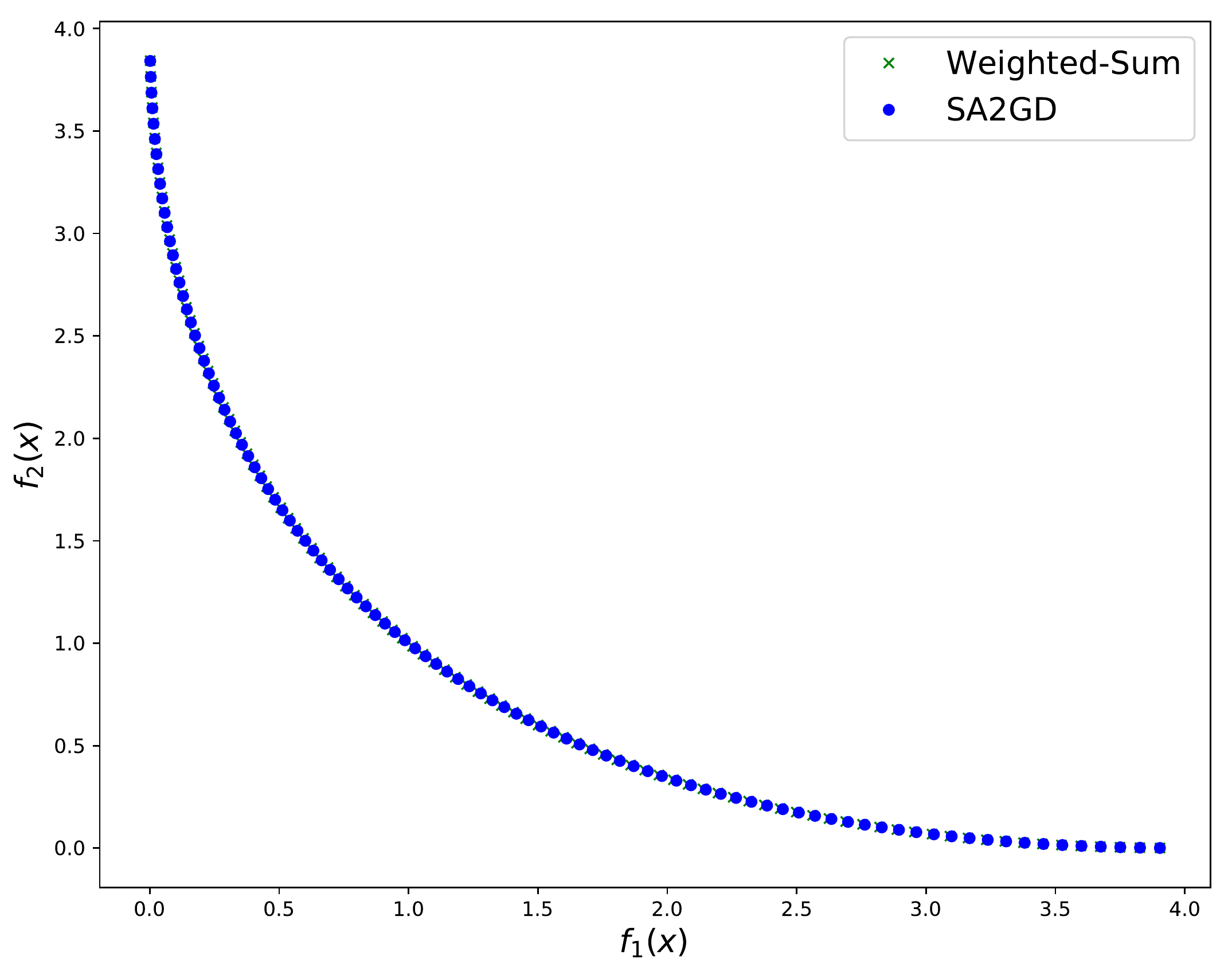}}\quad
   \subfloat[][Problem IM1.]{\includegraphics[width=.45\textwidth]{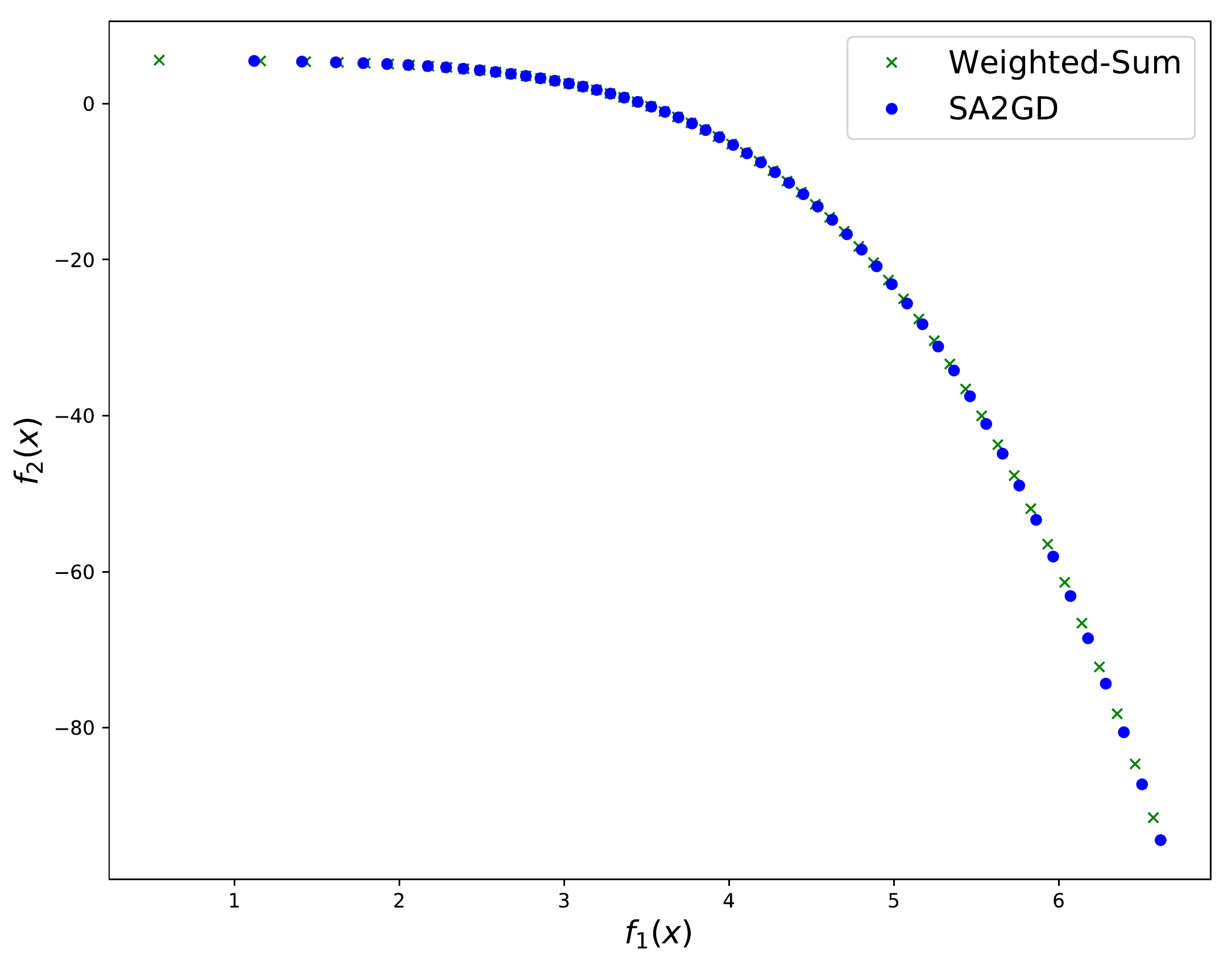}}\\
   \subfloat[][Problem MOP3.]{\includegraphics[width=.45\textwidth]{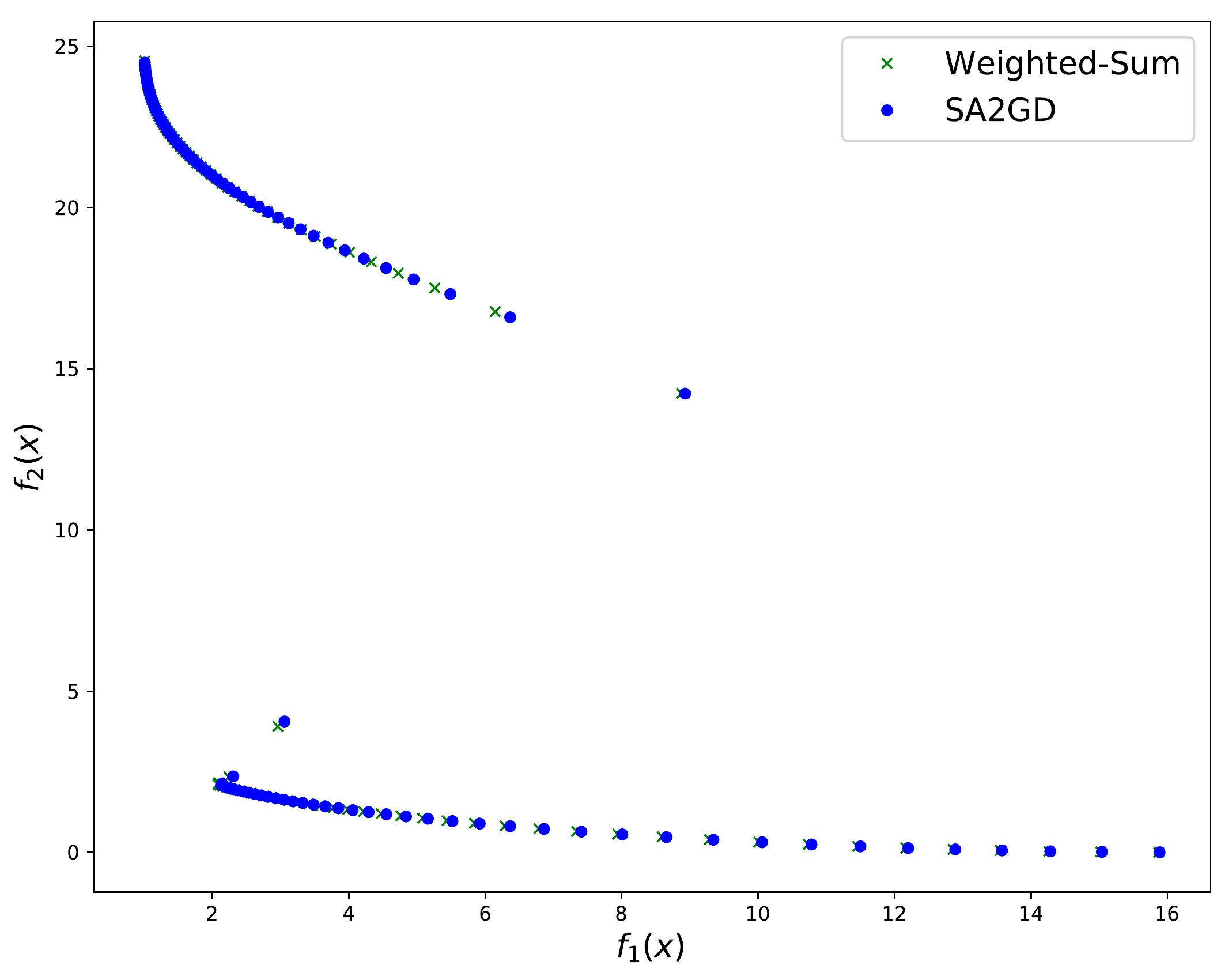}}\quad
   \subfloat[][Problem FAR1.]{\includegraphics[width=.45\textwidth]{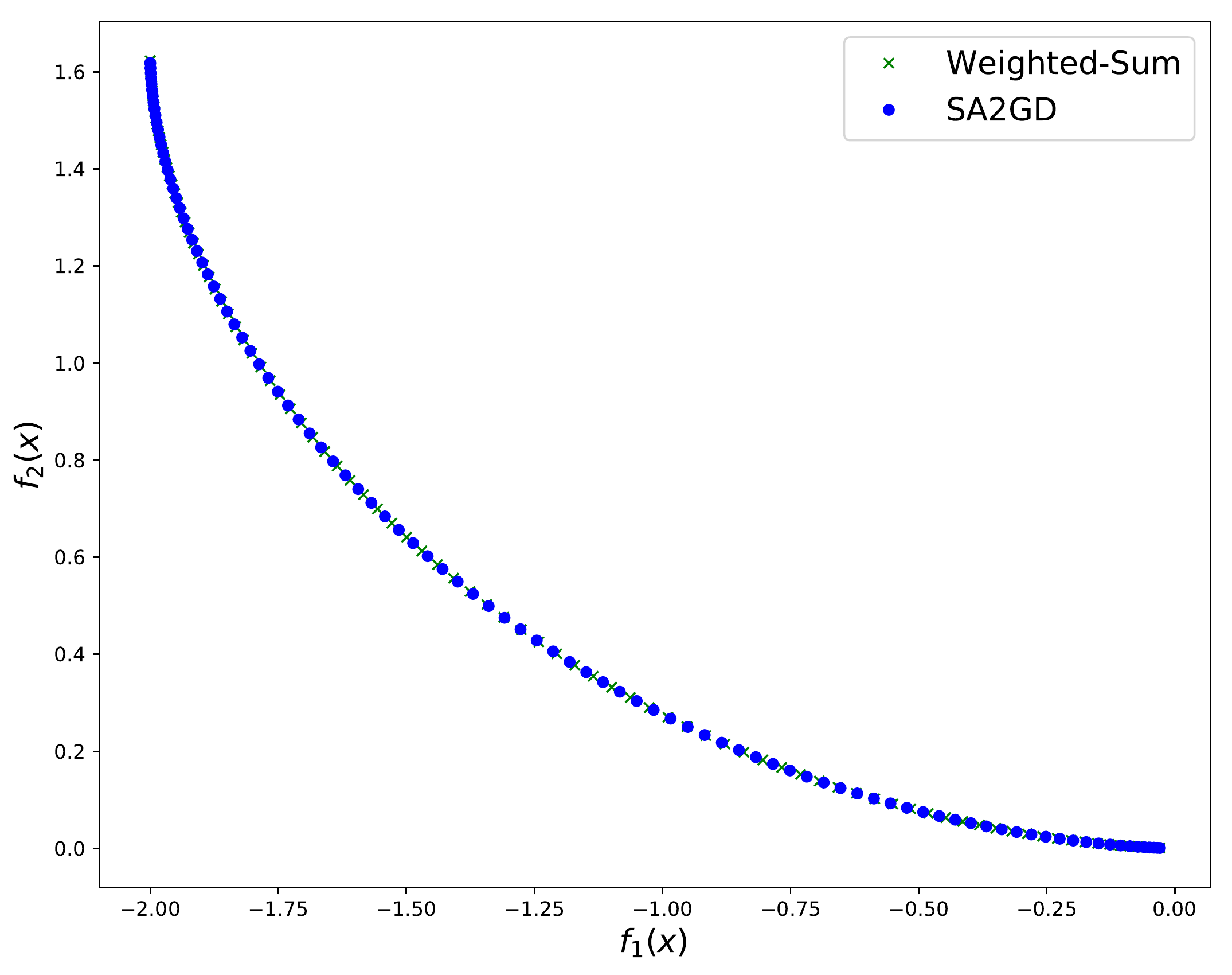}}
   \caption{An illustration of the scalarization by optimization effort of SA2GD when compared to the weighted-sum method.}
   \label{Pareto_fronts}
\end{figure}

\section{Conclusions}
\label{conclude}
We have shown that stochastic alternating bi-objective gradient or subgradient descent is convergent, achieving convergence rates
aligned with what is known for single-objective gradient (or subgradient) methods. The rates are $\mathcal{O}(1/\sqrt{T})$ for convex functions and $\mathcal{O}(1/T)$ for strongly convex functions, and they hold for both smooth and non-smooth functions. 

The analysis considered the case where all gradient (or subgradient) steps are applied first in a block to one function and then to the other also in block. We ingeniously applied the Intermediate Value Theorem to bundle all steps taken on each of the functions, a process that led to an implicit scalarization $(n_af^a+n_bf^b)/(n_a+n_b)$ of the two functions. It is however possible, following a similar proof logic, to show that the convergence rates are maintained regardless of the order according to which the gradient (subgradient) steps for the two objectives are executed. In fact, we can even, within a single iteration, randomly select the~$n_a$ positions where steps are applied to~$f^a$ from the $n_a + n_b$ total number of steps without impacting the rate of convergence. 

Furthermore, the proposed stochastic alternating optimization framework and theory could be generalized to more than two conflicting objectives. By varying the effort put in separately optimizing each objective within a convex linear combination, one can capture a trade-off among the multiple objectives.

In practice, many machine learning models lead to non-convex optimization problems. Non-convexity is present in the original bi-objective fair $k$-means clustering problem mentioned in the Introduction. A meaningful convergence guarantee for the non-convex case is certainly a topic for future research. 

\appendix
\section{Proposition using Intermediate Value Theorem}
\label{sec_prop_ivt}

Based on the Intermediate Value Theorem, we derive the following proposition for the purpose of convergence rate analysis of the SA2GD algorithm.
\begin{proposition} \label{prop_ivt}
Given a continuous real function $\phi(x): \mathbb{R}^n \to \mathbb{R}$ and a set of points $\{x_j\}_{j =1}^m$, there exists $w \in \mathbb{R}^n$ such that 
\[m\phi(w) = \sum_{j = 1}^m \phi(x_j),\]
where~$w = \sum_{j = 1}^m \mu_j x_j,\mbox{ with }\sum_{j = 1}^m \mu_j = 1, \mu_j \geq 0$, $i = 1, \ldots, m$, is a convex linear combination of $\{x_j\}_{j =1}^m$. 
\end{proposition}

\begin{proof}
The proposition is obtained by consecutively applying the Intermediate Value Theorem to $\phi(x)$. First, for the pair of points $x_1$ and $x_2$, there exists a point $w_{12} = \mu_{12} x_1 + (1-\mu_{12}) x_2, \mu_{12} \in [0, 1]$, such that~$\phi(w_{12}) = (\phi(x_1) + \phi(x_2))/2$ according to the Intermediate Value Theorem, which implies that $\sum_{j = 1}^m \phi(x_j) = 2\phi(w_{12}) + \sum_{j = 3}^m \phi(x_j)$.  Then, there exists $w_{13} = \mu_{13} w_{12} + (1-\mu_{13}) x_3, \mu_{13} \geq 0$, such that $\phi(w_{13}) = (2\phi(w_{12}) + \phi(x_3))/3$ holds given that the average function value $(2\phi(w_{12}) + \phi(x_3))/3$ lies between $\phi(w_{12})$ and $\phi(x_3)$. Notice that~$w_{13}$ can also be written as convex linear combination of~$\{x_1, x_2, x_3\}$. The proof is concluded by continuing this process until $x_m$ is reached. 
\end{proof}

\small
\bibliographystyle{plain}
\bibliography{SA2SG}

\end{document}